\documentclass{amsart2010}
\usepackage{graphicx}
\usepackage{amsmath}
\usepackage{amsfonts}
\usepackage{amssymb}
\usepackage{xcolor}
\usepackage{biblatex}
\newtheorem{thm}{Theorem}

\newtheorem{lem}[thm]{Lemma}
\newtheorem{prop}[thm]{Proposition}
\theoremstyle{definition}
\newtheorem{df}[thm]{Definition}

\newtheorem*{thm*}{Theorem}
\newtheorem*{df*}{Definition}
\newtheorem*{lem*}{Lemma}

\theoremstyle{remark}

\numberwithin{equation}{section}

\newcommand{\Ber}{\mathcal{A}^{2}_{\alpha}(\mathbb C_+)}
\newcommand{\Berd}{\mathcal{A}^{2}_{\alpha}(\mathbb D)}
\newcommand{\re}{{\rm Re}\,}
\newcommand{\im}{{\rm Im}\,}
\newcommand{\spa}{{\rm span}\,}

\title{Cyclic phenomena for composition operators on weighted Bergman spaces}

\author{Artur Blois }
\address[A. Blois]{IMECC, Universidade Estadual de Campinas, Campinas, Brazil}
\email{blois@ime.unicamp.br}



\begin{document}

\maketitle
\begin{abstract}
In this article, we study cyclic phenomena on the full range of weighted Bergman spaces on the right half-plane $\Ber.$ We completely characterize cyclic and supercyclic composition operators induced by affine self-maps of $\mathbb C_+,$ thereby completing the characterization of all hypercyclic, supercyclic and cyclic composition operators induced by affine symbols on $\Ber.$
\end{abstract}
\section{Introduction}
The study of cyclic phenomena has been one of the central research topics in linear dynamics in recent decades, more specifically, hypercyclic, supercyclic and cyclic operators. Let $\mathcal{X}$ be a separable Banach space. If $T$ is a linear operator acting on  $\mathcal{X},$ we say that $T$ is \textit{hypercyclic} if there exists a non-zero vector $f \in \mathcal{X}$ such that the orbit $\{T^nf: n \in \mathbb N\}$ is dense $\mathcal{X}.$ The first example of hypercyclic operator on a Banach space is due to Rolewicz \cite{Rolewicz} on the sequence space $\ell^p,$ for $1\leq p < \infty.$ Prior to Rolewicz, the only examples of hypercyclic operators were on the Fréchet space of entire functions due to Birkhoff and MacLane \cite{Birkhoff,Maclane}. Since then, many advances have been made on the topic and we suggest to the reader the recent monograph \cite{Bayart-Matheron} for a more in-depth view of hypercyclic operators. 

A weaker concept of hypercyclic operators are the supercyclic operators. Recall that if $T$ is linear operator on $\mathcal{X},$ it is said to be \textit{supercyclic} if there exist a non-zero vector $f \in \mathcal{X}$ such that the projective orbit $\{\lambda T^n f: n \in \mathbb N\,, \lambda \in \mathbb C\}$ is dense in $\mathcal{X}.$ It is easy to see that every hypercyclic operator is supercyclic, however the converse is not in general true. It was first introduced by Hilden and Wallen in \cite{HW} and since then many researchers have been interested in the subject, for instance, \cite{Ansari-Bourdon,BMhyponormal,GG-MR,Kumar,MR-Shkarin}. Another important definition is of $n$-\textit{supercyclicity}, which instead of asking for the projective orbit of vector being dense, we ask for the projective orbit of a $n$-dimensional vector space, for $n\geq 2,$ to be dense in $\mathcal{X}.$ Is worth noting that every supercyclic operator is $n$-supercylic but there exists a $n$-supercyclic operator which is not supercyclic \cite{Feldman}.  

Also, a linear operator $T$ is said to be \textit{cyclic} on $\mathcal{X}$ if there exists a non-zero vector $f$ such that $\spa \{T^nf:n \in \mathbb N\}$ is dense in $\mathcal{X}.$ Is easy to see that every supercyclic operator is cyclic, but again the converse is not true in general.

In the setting of spaces of holomorphic functions, one of the most studied classes of operators is of composition operators. Let $\Omega \subset \mathbb C$ be a domain in the complex plane and let $\mathcal{S}(\Omega)$ be a space of holomorphic functions on $\Omega$. We define the composition operator with holomorphic symbol $\phi:\Omega \rightarrow \Omega$ by
\[C_{\phi}: \mathcal{S}(\Omega) \rightarrow \mathcal{S}(\Omega), \quad f \mapsto f \circ \phi, \quad f \in \mathcal{S}(\Omega).\]
While studying spaces of holomorphic functions, one of the most studied symbols are the linear fractional symbols, that is, functions of the form
\[\eta(z)=\dfrac{az+b}{cz+d},\]
whose coefficients satisfy $ad-bc \neq 0$ and $\eta(\Omega) \subseteq \Omega.$ For the weighted Dirichlet space of the disk (which includes the classical Hardy, Bergman and Dirichlet spaces), the recent monograph by Gallardo-Gutiérrez and Montes-Rodríguez \cite{GG-MRbook} completely characterize all cyclic phenomena for linear fractional symbols.

In \cite{Noor-Severiano}, Noor and Severiano studied hypercyclicity and cyclicity for composition operators induced by linear fractional symbols on the Hardy space of right half-plane, $H^2(\mathbb C_+),$ and showed that no composition operator is hypercyclic but in some cases they are cyclic, leaving the question open if there were any supercyclic composition operators. Recently, Kumar \cite{Kumar} showed that $H^2(\mathbb C_+)$ supports no supercyclic composition operator induced by linear fractional symbols, however due to minor technical mistake the conclusion is incorrect. The aim of this paper is to further extend the study of composition operators on $\Ber$ started in \cite{Blois-Severiano}, to correct and extend the proof of supercyclicity given in \cite{Kumar} and completely characterize $n$-supercyclicity and cyclicity among said operators.

This article is organized as follows: Section \ref{Preliminaires} will be dedicated to the preliminaries, Section \ref{supercyclicity} we characterize supercyclicity, Section \ref{nsupercyclicity} we characterize $n$-supercyclicity and Section \ref{cyclicity} we characterize cyclicity.

\section{Preliminaries}\label{Preliminaires}
Throughout this paper $\mathbb N$ denotes the set of all non-negative integers, $\mathbb{C_+}$ the open right-half plane $\{z \in \mathbb C:\re(z)>0\},$ $\mathbb D$ the open unit disk $\{z \in \mathbb C : |z|<1\},$ $\mathbb R$ denotes the set of the real numbers and $\mathbb R^+$ denotes the set of positive real numbers.

For $\alpha > -1,$ we define the weighted Bergman space over the right-half plane, denoted by $\Ber,$ as the space of all holomorphic functions $f$ on $\mathbb C_+$ such that the norm
\[\|f\|=\left(\int_{\mathbb C_+}|f(x+iy)|^2x^{\alpha}dxdy \right)^{1/2}\]
is finite. In particular, $\Ber$ is a reproducing kernel Hilbert space, where the family $\{k^{\alpha}_w:w \in \mathbb C_+\}$ given by
\[k^{\alpha}_w(s)=\dfrac{2^{\alpha}(\alpha+1)}{(s+\overline{w})^{(\alpha+2)}}, \quad s \in \mathbb C_+,\]
are the reproducing kernels. These kernels define a continuous evaluation functional for all $f \in \Ber.$ These spaces are closely related to the Hardy space over the right half-plane, in fact, it is well-known in the literature that the ordinary Hardy space $H^{2}(\mathbb C_+)$ can be seen as the ``limit case" as $\alpha \rightarrow -1$ (see for instance \cite{Elliot-Wynn}). Similarly, we define the weighted Bergman space over the unit disk, denoted by $\Berd,$ as the space of all holomorphic functions $f$ on $\mathbb D$ such that the norm
\[\|f\|=\left(\int_{\mathbb D}|f(z)|(1-|z|^2)^{\alpha}dA(z) \right)^{1/2}\]
is finite, where $dA$ denotes the normalized area measure on $\mathbb D.$
The weighted Bergman spaces have a canonical isomorphism between $\Berd$ and $\Ber$ given by 
\begin{equation}\label{isomorphism}
    J:\mathcal{A}^{2}_{\alpha}(\mathbb D) \rightarrow \Ber, \quad (Jf)(w)=f\left(\dfrac{w-1}{w+1}\right)\dfrac{2^{\alpha+1}}{(w+1)^{\alpha+2}}.
\end{equation}

There exists a version of the Paley-Wiener theorem for $\Ber:$ Let $\mu_\alpha$ be the weighted Lebesgue measure given by
\[d\mu_\alpha=\dfrac{\pi\Gamma(1+\alpha)}{2^{\alpha}t^{\alpha+1}}dt,\]
where $\Gamma$ denotes the usual gamma function and $dt$ is the usual Lebesgue measure on $\mathbb R^+.$ Let $\mathcal{L}$ denote the Laplace transform given by
\[(\mathcal{L}F)(w)=\int_{0}^{\infty}F(t)e^{-wt}dt, \quad w \in \mathbb C_+.\]
By \cite[Lemma 3.3]{Elliot-Wynn}, $\Ber$ is isometrically isomorphic to $L^2_{\alpha}:=L^2(\mathbb R^+,d\mu_\alpha)$ via the Laplace transform. So $f \in \Ber$ if and only if there exists $F \in L^2_{\alpha}$ such that
\begin{equation}
    f(w)=\int_{0}^{\infty}F(t)e^{-wt}dt, \quad w \in \mathbb C_+.
\end{equation}

Our focus on this article are the composition operators, Elliot and Wynn \cite{Elliot-Wynn} proved the following boundedness criterion for composition operators on $\Ber$ \cite[Theorem 3.4]{Elliot-Wynn}
\begin{thm}
Let $\phi$ be a holomorphic self-map of $\mathbb C_+,$ then $C_{\phi}$ is bounded on $\Ber$ if and only if $\phi$ has finite angular derivative at $\infty,$ that is, $\phi(\infty)=\infty$ and the non-tangential limit
\[\phi^{\prime}(\infty)=\lim_{w \rightarrow \infty}\dfrac{w}{\phi(w)}\]
exists and is finite.
\end{thm}
In general, the norm of composition operators are not completely understood and we mainly have upper bounds for it, however, the norm of $C_{\phi}$ is precisely given by $\|C_{\phi}\|=(\phi^{\prime}(\infty))^{(\alpha+2)/2}$ on $\Ber.$ The only linear fractional symbols that induce bounded composition operators are the symbols of the form
\begin{equation}\label{symbols}
    \phi(w)=aw+b, \quad \text{where $a>0$ and $\re(b)\geq 0$,}
\end{equation}
which for simplicity we will call them \emph{affine symbols} and operators induced by them \emph{affine composition operators}. Among the affine symbols, we may categorize them in the following manner
\begin{df*}
\textit{Let $\phi$ be an affine symbol:
\begin{itemize}
    \item If $a=1$ we say that $\phi$ is of parabolic type. If additionally $\re(b)=0,$ we say that $\phi$ is a parabolic automorphism.
    \item If $a\neq 1$ we say that $\phi$ is of hyperbolic type. If additionally $\re(b)=0,$ we say that $\phi$ is a hyperbolic automorphism.
    \item If $\phi$ is of hyperbolic type, we say that $\phi$ is of type I if $a \in (0,1)$ and $\phi$ is of type II if $a\in(1,\infty).$
\end{itemize}
}
\end{df*}
If $\phi$ is as \eqref{symbols}, we may compute the iterates of $\phi$ as follows
\begin{align}\label{iterates}
\phi^{[n]}(w)= \begin{cases}w+nb, & \text{if} \ a=1, \\ 
a^{n} w+\frac{\left(1-a^{n}\right)}{1-a}b,  & \text{if} \ a \neq 1, 
\end{cases}
\end{align}
hence $C^{n}_{\phi}=C_{\phi^{[n]}}.$ By this iterate behavior we observe
\begin{lem}\label{convergencesymbol}
Let $\phi$ be an affine symbol such that $a \in (0,1).$ Then $\phi^{[n]}\rightarrow b/(1-a)$ locally uniformly as $n \rightarrow \infty.$
\end{lem}
In \cite[Lemma 3.7]{Blois-Severiano}, the authors were able to show that the composition operator $C_{\phi}$ is unitarily equivalent to the weighted composition operator $\widehat{C}_{\phi}$ in $L^2_\alpha$ given by
\begin{equation}\label{equivalenceOperator}
    (\widehat{C}_{\phi}F)(t)=\dfrac{1}{a}e^{-bt}F(t/a), \quad F \in L^2_\alpha.
\end{equation}
\section{Supercyclicity}\label{supercyclicity}

We devote this section to the study of supercyclicity among affine composition operators on $\Ber.$ Here we correct the proof of \cite{Kumar} and further extend the characterization of supercyclicity for the weighted Bergman space and we start the section by stating its main result.
\begin{thm}\label{mainsuper}
Let $\phi$ be an affine symbol. Then $C_{\phi}$ is supercyclic on $\Ber$ if and only if $\phi$ is hyperbolic of type II.
\end{thm}

We divide the main result in a few steps, one of the key ingredients is the normality among said operators. By \cite[Proposition 3.8]{Blois-Severiano}, we know that the normality condition holds if and only if $a=1$ or $\re(b)=0.$ In \cite{HW} Hilden and Wallen proved that no normal operator can be supercyclic, so we are able to discard supercyclicity for some of our operators.

\begin{prop}
If $\phi$ is of parabolic type or is a hyperbolic automorphism, then $C_{\phi}$ is not supercyclic on $\Ber.$
\end{prop}
\begin{proof}
If $\phi$ is a hyperbolic automorphism, then we have $\re(b)=0$ which implies that $C_{\phi}$ is normal and therefore is not supercyclic on $\Ber.$ On the other hand, if $\phi$ is of parabolic type, then $a=1$ which again implies that $C_{\phi}$ is normal and the conclusion holds.
\end{proof}

Now we address the hyperbolic type symbols. Note that if $\phi$ is of type I, we know that $b/(1-a)$ is a fixed point of $\phi$ and we will make use of Lemma \ref{convergencesymbol}

\begin{prop}\label{hypI}
If $\phi$ is hyperbolic of type I, then $C_{\phi}$ is not supercyclic on $\Ber.$
\end{prop}
\begin{proof}
Suppose that there exists a supercyclic vector $f \in \Ber$ for $C_{\phi},$ then given any function $g \in \Ber$ there exists a sequence $(a_k)_{k \in \mathbb N} \subset \mathbb C$ and an increasing sequence $(n_k)_{k \in \mathbb N} \subset \mathbb N$ such that
\[\|a_{k}C^{n_k}_{\phi}f-g\|\rightarrow 0, \quad \text{as $k \rightarrow \infty$}.\]
In particular, since $\Ber$ is a reproducing kernel Hilbert space, by the Cauchy-Schwarz inequality we get
\[|a_kf(\phi^{[n_k]}(w))-g(w)|\leq \|a_kC^{n_k}_\phi f-g\|\|k^{\alpha}_{w}\|, \quad w \in \mathbb C_+.\]
Hence we obtain a point-wise convergence
\begin{equation}\label{convsuper}
    \lim_{k \rightarrow \infty}a_kf(\phi^{[n_k]}(w))=g(w), \quad w \in \mathbb C_+.
\end{equation}
By Lemma \ref{convergencesymbol}, we have a local uniform convergence of $\phi^{[n_k]},$ so up to a subsequence we note that the left-hand side of \eqref{convsuper} is independent of $w,$ thus the function $g$ is constant and therefore it is the null function since $\Ber$ does not admit any other constant function. However, this is a contradiction to the assumption that $f$ is a supercyclic vector and we conclude that $C_{\phi}$ does not admit a supercyclic vector on $\Ber.$
\end{proof}
Note that if we take a non-automorphic self-map of $\mathbb C_+$ with a fixed point in $\mathbb C_+$, the proof of Proposition \ref{hypI} still holds. Now, we need only to prove that every hyperbolic of type II symbol induce a supercyclic composition operator on $\Ber.$ In \cite[Theorem 3.20]{Kumar}, it claims that on $H^2(\mathbb C_+)$ a composition operator induced by a hyperbolic of type II symbol is a scalar multiple of a composition operator induced by hyperbolic of type I, which is true to its adjoint as is stated in \cite[Proposition 1]{Noor-Severiano}.
\begin{prop}\label{super}
If $\phi$ is hyperbolic of type II, then $C_{\phi}$ is supercyclic on $\Ber.$   
\end{prop}

\begin{proof}
It is known in the literature that the operator $C_{\phi}$ is similar to $C_{\psi_a},$ where $\psi_a(w)=aw+a(1-a^{-1}).$ Recall the isomorphism given in \eqref{isomorphism}, a straightforward calculation shows that
    \[J(C_{\varphi_af})(w)=\left(\dfrac{\psi_a(w)+1}{w+1}\right)^{\alpha+2}C_{\psi_a}(Jf)(w),\]
where $\varphi_a$ is given by $\varphi_a(z)=a^{-1}z+1-a^{-1}.$ Thus, $C_{\varphi_a}$ is similar to a $a^{\alpha+2}C_{\psi_a}.$ By \cite[Proposition 5.2]{FRhyper}, it follows that $C_{\varphi_a}$ is hypercyclic on $\Berd.$ Via the isomorphism $J,$ we get that $a^{\alpha+2}C_{\psi_a}$ is hypercyclic on $\Ber$ and in particular it is supercyclic. We use a known fact that if a linear operator is supercyclic, any non-null scalar multiple is supercyclic, from which we conclude that $C_{\psi_a}$ is supercyclic and by similarity, $C_{\phi}$ is supercyclic, completing the proof.
\end{proof}
The proofs above can be adapted to the Hardy space $H^2(\mathbb C_+)$ setting by simply using the characterization of normality given in \cite[Theorem 2]{Noor-Severiano}, the fixed point argument also holds with the same arguments and finally by using the canonical isomorphism between $H^{2}(\mathbb D)$ and $H^2(\mathbb C_+)$ (see \cite{Noor-Carmo} for a discussion in this subject) and the fact that $C_{\varphi_a}$ is hypercyclic in $H^{2}(\mathbb D)$ (see \cite{GG-MRbook} or \cite{Shapiro}). Therefore, the classical cyclic phenomena for the Hardy space $H^2(\mathbb C_+)$ is now completely characterized for affine composition operators, completing the work of Noor and Severiano \cite{Noor-Severiano}.
\section{$n$-supercyclicity}\label{nsupercyclicity}
In this section we completely characterize $n$-supercyclicity for affine composition operators on $\Ber.$ For this goal we first characterize hyponormality among these operators
\begin{df}
Let $T$ be a linear operator on a Hilbert space $\mathcal{H}.$ We say that $T$ is hyponormal if $T^*T \geq TT^*.$
\end{df}
\begin{lem}\label{hypo}
    Let $\phi$ be an affine symbol. Then $C_{\phi}$ is hypornormal if and only if $a \leq 1$ or $\re(b)=0.$
\end{lem}
\begin{proof}
By the Paley-Wiener theorem, $\Ber$ is isometrically isomorphic to $L^{2}_{\alpha}$ and $C_{\phi}$ is equivalent to a weighted composition operator $\widehat{C}_{\phi}$ given by equation \eqref{equivalenceOperator} and also by \cite[Proposition 3.8]{Blois-Severiano} we have the following
    \[(\widehat{C}_{\phi}^*\widehat{C}_{\phi}F)(t)=a^{-(\alpha+2)}e^{-2\re(b)t}F(t) \quad \text{and} \quad (\widehat{C}_{\phi}\widehat{C}_{\phi}^*F)(t)=a^{-(\alpha+2)}e^{-2\re(b)t/a}F(t),\]
    therefore the condition of $\widehat{C}_{\phi}^*\widehat{C}_{\phi} \geq \widehat{C}_{\phi}\widehat{C}_{\phi}^*$ is equivalent to
    \[e^{-2\re(b)t}\geq e^{-2\re(b)t/a}.\]
The above relation is valid if and only if $\re(b)=0$ or $a\leq1.$
\end{proof}
 In \cite[Theorem 3.2]{BMhyponormal}, Bayart and Matheron proved that no hyponormal operator can be $n$-supercyclic. Thus with the characterization of hyponormality, we are able to conclude the following
\begin{thm}
Let $\phi$ be an affine symbol. Then $C_{\phi}$ is $n$-supercyclic on $\Ber$ if and only if $\phi$ is hyperbolic of type II.
\end{thm}
\begin{proof}
By Lemma \ref{hypo}, every symbol except a hyperbolic of type II is hyponormal and therefore cannot be $n$-supercyclic. By Theorem \ref{mainsuper}, only hyperbolic of type II symbols are $n$-supercyclic since they are supercyclic.
\end{proof}
\section{Cyclicity}\label{cyclicity}
In this section, we characterize cyclicity among affine composition operators. Some of the techniques used in this section are adapted arguments from \cite{GG-MRAdjoints,Noor-Severiano}. We start this section by presenting an adjoint formula for the affine composition operators only in terms of operators in $\Ber.$ 
\begin{prop}
Let $\phi$ be an affine symbol. Then the adjoint of $C_{\phi}$ acting on $\Ber$ is given by
\[C^{*}_{\phi}=a^{-(\alpha+2)}C_{\psi},\]
where $\psi(w)=a^{-1}w+a^{-1}\overline{b}.$
\end{prop}
\begin{proof}
We know that if $\phi$ is an affine symbol, then $C_{\phi}$ is unitarily equivalent to the operator
\[(\widehat{C}_{\phi}F)(t)=\dfrac{1}{a}e^{-bt/a}F(t/a)\]
acting on $L^2_{\alpha}.$ Now the adjoint in $L^{2}_{\alpha}$ is given by
\[(\widehat{C}^{*}_{\phi}F)(t)=a^{-(\alpha+1)}e^{-\overline{b}t}F(at).\]
A straightforward calculation shows that
\[\widehat{C}^{*}_{\phi}\mathcal{L}=\mathcal{L}(a^{-(\alpha+2)}C_{\psi}),\]
where $\psi(w)=a^{-1}w+a^{-1}\overline{b}.$
\end{proof}
Its clear that if $C_{\phi}$ is induced by a hyperbolic of type I symbol, its adjoint is a scalar multiple of a composition operator induced by a hyperbolic of type II symbol and vice versa. We first characterize cyclicity for operators induced by hyperbolic type symbols. In \cite[Proposition 2.7]{Bourdon-Shapiro}, Bourdon and Shapiro proved that if the adjoint of a linear operator $T$ on a Hilbert space $\mathcal{H}$ has a multiple eigenvalue, then $T$ is not cyclic.
\begin{prop}
If $\phi$ is hyperbolic of type I, then $C_{\phi}$ is not cyclic on $\Ber.$
\end{prop}
\begin{proof}
Take the function 
\[f_{\lambda}(w)=\left(w+\dfrac{b}{a-1}\right)^{\lambda}\]
where $a > 1$ and $\re(b)>0.$ Is easy to see that $f_{\lambda} \in \Ber$ if $\re(\lambda) < \dfrac{-(\alpha+2)}{2}.$ We want to show that functions of the form $f_{\lambda}$ are eigenvectors for a composition operator induced by a hyperbolic of type II symbol. Let $\psi(w)=aw+b$ denote a hyperbolic of type II symbol, thus
\[(C_{\psi}f_{\lambda})(w)=\left(aw+b+\dfrac{b}{a-1}\right)^{\lambda}=a^{\lambda}f_\lambda(w).\]
Now we show that $a^{\lambda}$ is an eigenvalue of infinite multiplicity. Let $\lambda_n=\lambda+\frac{2\pi n}{\log a}i$ for $n \in \mathbb N$ and we get
\[C_{\psi}f_{\lambda_n}=a^{\lambda_n}f_{\lambda_n}=a^{\lambda}f_{\lambda_n}, \quad \forall \, n \in \mathbb N.\]
Since $(f_{\lambda_n})_{n \in \mathbb N}$ is an infinite family of functions, follows that $a^{\lambda}$ is an eigenvalue of infinite multiplicity for $C_{\psi}.$ Therefore the adjoint of $C_{\phi}$ has a multiple eigenvalue and thus $C_{\phi}$ cannot be cyclic.
\end{proof}
Now, from Proposition \ref{super}, its clear that if $\phi$ is hyperbolic of type II then $C_{\phi}$ is cyclic on $\Ber.$ Now we address the automorphic case, which we adapt techniques used by Gallardo-Gutiérrez and Montes-Rodríguez in \cite{GG-MRAdjoints}.
\begin{prop}
If $\phi$ is a hyperbolic automorphism, then $C_{\phi}$ is not cyclic on $\Ber.$
\end{prop}
\begin{proof}
If $\phi$ is a hyperbolic automorphism, then $a\neq 1$ and $\re(b)=0.$ We consider the parabolic automorphism given by
\[\tau(w)=w+\dfrac{b}{1-a}\]
and we note that $C^{-1}_{\tau}C_{\phi}C_{\tau}=C_{\varphi},$ where $\varphi(w)=aw.$ Therefore $C_{\phi}$ is similar to $C_{\varphi}$ and by Laplace transform we obtain that $C_{\varphi}$ is equivalent to $\widehat{C}_{\varphi}$ as in \eqref{equivalenceOperator}. 

Now define the linear operator given by
\[(TF)(t)=\left(\dfrac{\pi\Gamma(1+\alpha)}{2^{\alpha}t^{\alpha}}\right)^{1/2}F(t)\]
and we note that the operator $T$ induces an isometry from $L^2_{\alpha}$ onto $L^2(\mathbb R^{+},dt/t).$ In particular, one can show that
\[T\widehat{C}_{\varphi}=V_aT,\]
where $V_a$ is a linear operator on $L^{2}(\mathbb R^+,dt/t)$ given by
\[(V_aF)(t)=a^{\frac{-(\alpha+2)}{2}}F(t/a).\]
If we consider the Fourier transform with respect to the multiplicative group of positive real numbers, namely
\[\hat{F}(t)=\dfrac{1}{2\pi}\int_{0}^{\infty}F(x)x^{-it}\dfrac{dx}{x},\]
it defines an isometry from $L^2(\mathbb R^+,dt/t)$ onto $L^2(\mathbb R, 2\pi dt)$ and the operator $V_a$ is unitarily similar to the operator $a^{\frac{-(\alpha+2)}{2}}M_{\eta},$ where $M_\eta$ is the multiplication operator with symbol $\eta(t)=a^{-it},$ acting on $L^2(\mathbb R,2\pi dt).$ So far we showed that the operator $C_{\phi}$ is unitarily similar to $a^{\frac{-(\alpha+2)}{2}}M_{\eta}$ and now suffices to show that $M_{\eta}$ is not cyclic. We may rewrite $\eta$ in the following fashion
\[a^{-it}=e^{\log a^{-it}}=e^{-it\log a}.\]
In \cite[Lemma 8]{Noor-Severiano}, Noor and Severiano showed that a multiplication operator induced by a symbol of the form $e^{ist},$ with $s \in \mathbb R,$ is not cyclic on $L^2(\mathbb R, 2\pi dt)$ and hence $C_{\phi}$ cannot be cyclic on $\Ber.$
\end{proof}

Finally, we characterize cyclicity among composition operators induced by symbols of parabolic type
\begin{prop}
If $\phi$ is of parabolic type, then $C_{\phi}$ is cyclic on $\Ber$ if and only if $\re(b)>0.$
\end{prop}
\begin{proof}
By equation \eqref{equivalenceOperator}, we have that $C_{\phi}$ in unitarily equivalent to $\widehat{C}_{\phi},$ which is given by 
\[(\widehat{C}_{\phi}F)(t)=e^{-bt}F(t)\]
and we may rewrite $e^{-bt}$ as
\[e^{-\re(b)t}e^{-i\im(b)t}.\]
By \cite[Theorem 3.3]{Seid}, $\widehat{C}_{\phi}$ is cyclic if and only $e^{-bt}$ is essentially univalent. Is easy to see that $e^{-bt}$ is essentially univalent if and only if $\re(b)>0$ (otherwise it would a periodic function), thus the result holds. 
\end{proof}

\section*{Acknowledgments}
This study was financed in part by the Coordenação de Aperfeiçoamento de Pessoal de Nível Superior – Brasil (CAPES) – Finance Code 001.

\printbibliography
\end{document}